\documentclass[11pt,a4paper]{article}
\usepackage{amsfonts,amsgen,amstext,amsbsy,amsopn,amsfonts,amssymb,amscd}
\usepackage[leqno]{amsmath}
\usepackage[amsmath,amsthm,thmmarks]{ntheorem}
\usepackage{epsf,epsfig}
\usepackage{float}
\usepackage{dsfont}
\usepackage{ebezier,eepic}
\usepackage{color}
\usepackage{booktabs}
\usepackage{tikz}
\usepackage{multirow}
\usepackage{mathrsfs}
\usepackage{graphicx}
\usepackage{subfigure}
\newtheorem{thm}{Theorem}[section]

\newtheorem{prop}[thm]{Proposition}

\newtheorem{lem}[thm]{Lemma}
\newtheorem{example}[thm]{Example}
\newtheorem{false statement}{False statement}
\newtheorem{cor}[thm]{Corollary}

\theoremstyle{definition}

\newtheorem{claim}[thm]{Claim}

\newtheorem{conj}[thm]{Conjecture}

\makeatletter \@addtoreset{equation}{section}

\def\hh{\mathcal{H}}

\def\hht{\mathcal{T}}

\def\hf{\mathcal{F}}
\def\hg{\mathcal{G}}

\def\hb{\mathcal{B}}

\begin{document}

\title{\bf\Large Four-vertex traces of finite sets}
\date{}
\author{Peter Frankl$^1$, Jian Wang$^2$\\[10pt]
$^{1}$R\'{e}nyi Institute, Budapest, Hungary\\[6pt]
$^{2}$Department of Mathematics\\
Taiyuan University of Technology\\
Taiyuan 030024, P. R. China\\[6pt]
E-mail:  $^1$frankl.peter@renyi.hu, $^2$wangjian01@tyut.edu.cn
}
\maketitle

\begin{abstract}
Let $[n]=X_1\cup X_2\cup X_3$ be a partition with $\lfloor\frac{n}{3}\rfloor \leq |X_i|\leq \lceil\frac{n}{3}\rceil$ and  define $\mathcal{G}=\{G\subset [n]\colon |G\cap X_i|\leq 1, 1\leq i\leq 3\}$. It is easy to check that the trace $\mathcal{G}_{\mid Y}:=\{G\cap Y\colon G\in \mathcal{G}\}$ satisfies $|\mathcal{G}_{\mid Y}|\leq 12$ for all 4-sets $Y\subset [n]$. For $n\geq 25$ it is proven that whenever $\mathcal{F}\subset 2^{[n]}$ satisfies $|\mathcal{F}|>|\mathcal{G}|$ then $|\mathcal{F}_{\mid C}|\geq 13$ for some $C\subset [n]$, $|C|=4$. Several further results of a similar flavor are established as well.
\end{abstract}

\section{Introduction}

Let $[n]=\{1,2,\ldots,n\}$ be the standard $n$-element set, $2^{[n]}$ its powerset. For a family $\hf\subset 2^{[n]}$ and  a subset $Y\subset [n]$ let $\hf_{\mid Y}=\{F\cap Y\colon F\in\hf\}$ denote the {\it trace} of $\hf$ on $Y$. Hajnal \cite{bondy} introduced the {\it arrow relation} $(n,m)\rightarrow (a,b)$ to denote that for all $\hf\subset 2^{[n]}$ with $|\hf|\geq m$ there exists an $a$-element set $Y\subset  [n]$ such that $|\hf_{\mid Y}|\geq b$. For $\hf\subset 2^{[n]}$,  let $\hf\rightarrow (a,b)$ denote that there exists an $a$-element set $Y\subset  [n]$ such that $|\hf_{\mid Y}|\geq b$.

One of the most important results in extremal set theory, the Sauer-Shelah-Vapnik-Chervonenkis Theorem (\cite{S},\cite{SP},\cite{VC}) is equivalent to the arrow relation
\begin{align}\label{ineq-1.1}
\left(n,1+\sum_{i<k}\binom{n}{i}\right)\rightarrow (k,2^k) \mbox{ for all } n\geq k\geq 0.
\end{align}

Lov\'{a}sz \cite{L} conjectured and the first author \cite{F83} proved
\begin{align}\label{ineq-1.2}
\left(n,\left\lfloor \frac{n^2}{4}\right\rfloor+n+2\right)\rightarrow (3,7).
\end{align}

A family $\hf$ is called a {\it down-set} (or {\it complex}) if $F\in \hf$ always implies $2^F\subset \hf$. Both the above results are direct consequences of the following

\begin{lem}\label{lem-1.1}
If $\hf\not \rightarrow (a,b)$ for some family $\hf\subset 2^{[n]}$ then there is a down-set with the same property.
\end{lem}

\begin{example}
Let $\ell$ be a positive integer and $[n]=X_0\cup \ldots\cup X_{\ell-1}$ a partition with $|X_i|=\left\lfloor\frac{n+i}{\ell}\right\rfloor$, $0\leq i< \ell$.  Define
\[
\hf(n,\ell)=\{F\subset [n]\colon|F\cap X_i|\leq 1, 0\leq i<\ell\}.
\]
Clearly, $|\hf(n,\ell)|=\prod\limits_{0\leq i<\ell}\left(1+\left\lfloor\frac{n+i}{\ell}\right\rfloor\right)$ and for $Y\in \binom{[n]}{\ell+1}$, $|\hf(n,\ell)_{\mid Y}|\leq 3\cdot 2^{\ell-1}$ is easy to verify.

In particular, $|\hf(n,2)|=\lfloor\frac{n^2}{4}\rfloor+n+1$ shows that $(n,\lfloor\frac{n^2}{4}\rfloor+n+1)\not\rightarrow (3,7)$, i.e., the corresponding arrow relation does not hold.

For general $\ell$, the example shows that
\begin{align}\label{ineq-1.3}
\left(n,\prod_{0\leq i<\ell} \left\lfloor\frac{n+\ell+i}{\ell}\right\rfloor\right)\not\rightarrow (\ell+1,3\cdot 2^{\ell-1}+1).
\end{align}
As  \eqref{ineq-1.1} and \eqref{ineq-1.2} show \eqref{ineq-1.3} is best possible for $\ell=1$ and 2.
\end{example}

It is very limited evidence but let us make a conjecture for the general case.

\begin{conj}
\begin{align}\label{ineq-1.4}
\left(n,1+\prod_{0\leq i<\ell} \left\lfloor\frac{n+\ell+i}{\ell}\right\rfloor\right)\rightarrow (\ell+1,3\cdot 2^{\ell-1}+1)\mbox{ for  all } n>\ell>0.
\end{align}
\end{conj}

As we will see in the next section, \eqref{ineq-1.4} is closely related to some classical results. Our main result settles the $\ell=3$ case for $n\geq 25$.

\begin{thm}\label{thm-main}
\eqref{ineq-1.4} holds for $\ell=3$ and $n\geq 25$.
\end{thm}

In view of Lemma \ref{lem-1.1} to check the veracity of \eqref{ineq-1.4} we can restrict ourselves to down-sets. Moreover, we may assume that $\hf\subset 2^{[n]}$ contains no members of size exceeding $\ell$. We shall use these facts without further mention.

We need the following notations:
$$\hf(i)=\{F\setminus\{i\}\colon i\in F\in \hf\}, \ \hf(\bar{i})= \{F\in\hf: i\notin F\}.$$
Note that $|\hf|=|\hf(i)|+|\hf(\bar{i})|$. For $i,j\in [n]$, we also use
\[
\hf(i,j)=\{F\setminus\{i,j\}\colon \{i,j\}\subset F\in \hf\}, \ \hf(\bar{i},\bar{j})=\{F\in \hf\colon F\cap \{i,j\}=\emptyset\}.
\]
For $\hf\subset 2^{[n]}$, let $\hf^{(\ell)}$ denote the subfamily $\{F\in \hf\colon |F|=\ell\}$.

\section{Cancellative families}

Let us recall that an $\ell$-graph $\hh\subset \binom{[n]}{\ell}$ is called {\it cancellative} if $\hh$ contains no three edges $H_1,H_2,H_3$ such that $|H_1\cap H_2|=\ell-1$ and  $H_1\bigtriangleup H_2\subset H_3$ where $\bigtriangleup$ denotes the symmetric difference.

\begin{claim}\label{claim-2.1}
If $\hf\subset 2^{[n]}$ is a down-set and $\hf^{(\ell)}$ is not cancellative, then $\hf\rightarrow (\ell+1,3\cdot 2^{\ell-1}+1)$.
\end{claim}
\begin{proof}
Choose $F_1,F_2,F_3\in \hf^{(\ell)}$ such that $|F_1\cap F_2|=\ell-1$ and $F_1\bigtriangleup F_2 \subset F_3$. Set $Y=F_1\cup F_2$. Then $|Y|=\ell+1$ and both $2^{F_1}$ and $2^{F_2}$ are contained in $\hf_{\mid Y}$. Note that $|2^{F_1}\cup 2^{F_2}|=2\cdot 2^{\ell}-2^{\ell-1}=3\cdot 2^{\ell-1}$. Since the 2-element set $F_1\bigtriangleup F_2$ is in $2^{Y}\setminus (2^{F_1}\cup 2^{F_2})$ and $F_1\bigtriangleup F_2\subset F_3$, $F_1\bigtriangleup F_2\in \hf_{\mid Y}$ as well. Thus $|\hf_{\mid Y}|\geq 3\cdot 2^{\ell-1}+1$.
\end{proof}

The following statement was proved for $\ell=2$ by Mantel \cite{M}, for $\ell=3$ by Bollob\'{a}s \cite{bollobas} and for $\ell=4$ by Sidorenko \cite{Si}.

\begin{thm}
Let $2\leq \ell \leq 4$ and $\hh\subset \binom{[n]}{\ell}$. If $\hh$ is cancellative then
\begin{align}\label{ineq-1.5}
|\hf| \leq \prod_{0\leq i<\ell} \left\lfloor \frac{n+i}{\ell} \right\rfloor.
\end{align}
\end{thm}

Let us suppose that $\ell=3$ and $\hf\subset 2^{[n]}$ is a down-set with $\hf\not\rightarrow (4,13)$. Then $\hf^{(k)}=\emptyset$ for $k\geq 4$ and by \eqref{ineq-1.5} $|\hf^{(3)}|\leq \lfloor\frac{n+2}{3}\rfloor\lfloor\frac{n+1}{3}\rfloor\lfloor\frac{n}{3}\rfloor$. Consequently,
\[
|\hf| \leq \left\lfloor\frac{n+2}{3}\right\rfloor\left\lfloor\frac{n+1}{3}\right\rfloor\left\lfloor\frac{n}{3}\right\rfloor
+\binom{n}{2}+\binom{n}{1}+\binom{n}{0}.
\]
That is,
\begin{align}
\left(n,\left\lfloor\frac{n+2}{3}\right\rfloor\left\lfloor\frac{n+1}{3}\right\rfloor\left\lfloor\frac{n}{3}\right\rfloor
+\binom{n}{2}+n+2\right)\rightarrow (4,13).
\end{align}
This shows that \eqref{ineq-1.4} is ``asymptotically" true for $\ell=3$.

Similarly, the $\ell=4$ case of \eqref{ineq-1.5} and Lemma \ref{lem-1.1} imply
\begin{align}
\left(n,\prod_{0\leq i<4}\left\lfloor\frac{n+i}{4}\right\rfloor+\binom{n}{3}+\binom{n}{2}+n+2\right)\rightarrow (5,25).
\end{align}

Unfortunately, \eqref{ineq-1.5} is no longer true for $\ell\geq 5$. In particular for $\ell=5$ and 6 Frankl and F\"{u}redi \cite{FF} showed that the maximum possible size $m(n,\ell)$ of a cancellative family $\hf\subset \binom{[n]}{\ell}$ satisfies
\begin{align*}
&m(n,5) \leq \frac{6}{11^4}n^5 \mbox{ with equality iff } 11|n \mbox{ and }\\[5pt]
&m(n,6) \leq \frac{11}{12^5}n^6 \mbox{ with equality iff } 12|n,
\end{align*}
which is much larger than $(n/\ell)^{\ell}$.

Let us define $m^*(n,\ell)$ as the maximum size of $\hf\subset \binom{[n]}{\ell}$ where $\hf$ contains no three distinct edges satisfying $F_1\bigtriangleup F_2\subset F_3$. Unlike with cancellative families, we do not require $|F_1\cap F_2|=\ell-1$. Thus $m^*(n,\ell)\leq m(n,\ell)$.

Katona conjectured $m^*(n,\ell)=\prod\limits_{0\leq i<\ell} \left\lfloor\frac{n+i}{\ell}\right\rfloor$. However, Shearer \cite{Sh} disproved this conjecture for $\ell>10$.

%

\section{Proof of Theorem \ref{thm-main}}

We need the following inequality.

\begin{lem}
Let $a_1,a_2,\ldots,a_m\geq 0$. Then
\begin{align}\label{ineq-key0}
\prod_{1\leq i<j\leq m} a_ia_j \leq \frac{m-1}{2m}\left(\sum_{1\leq i\leq m}a_i\right)^2.
\end{align}
\end{lem}
\begin{proof}
Note that
\begin{align}\label{ineq-3.3}
\prod_{1\leq i<j\leq m} a_ia_j =\frac{1}{2}\left(\left(\sum_{1\leq i\leq m}a_i\right)^2-\sum_{1\leq i\leq m} a_i^2\right).
\end{align}
Since $x^2$ is convex, by    Jensen's inequality
\[
\frac{1}{m}\sum_{1\leq i\leq m} a_i^2\geq \left(\frac{1}{m}\sum_{1\leq i\leq m} a_i\right)^2=\frac{1}{m^2}\left(\sum_{1\leq i\leq m} a_i\right)^2.
\]
It follows that
\[
\sum\limits_{1\leq i\leq m} a_i^2 \geq \frac{1}{m}\left(\sum\limits_{1\leq i\leq m} a_i\right)^2.
\]
By \eqref{ineq-3.3} we conclude that \eqref{ineq-key0} holds.
\end{proof}

\begin{proof}[Proof of Theorem \ref{thm-main}]
Let $\hf\subset 2^{[n]}$ be a down-set satisfying  $\hf\not \rightarrow (4,13)$ and $|\hf|$ is maximal. Clearly, $|\hf|\geq |\hf(n,3)|= \lfloor\frac{n+3}{3}\rfloor\lfloor\frac{n+4}{3}\rfloor\lfloor\frac{n+5}{3}\rfloor$.

We showed that $\hf^{(3)}$ is cancellative however we are not going to use the bound \eqref{ineq-1.5}.

\begin{claim}\label{claim-3.1}
Let $\hf'$ be a family obtained from $\hf$ by removing all edges $F\in \hf$ with $y\in F$ and adding the edges $\{y\}\cup G$ for $G\in \hf(x,\bar{y})$. Then $\hf'\not\rightarrow (4,13)$.
\end{claim}
\begin{proof}
 Indeed, otherwise let $C$ be a 4-set satisfying $|\hf'_{\mid C} |\geq 13$. Then clearly $y\in C$. If $x\in C$, then by $\hf'(x,y)=\emptyset$
\[
|\hf'_{\mid C} | \leq 2^{|C\setminus\{x\}|}+2^{|C\setminus\{y\}|}-2^{|C\setminus\{x,y\}|}=2^3+2^3-2^2=12,
\]
a contradiction. Thus $x\notin C$. Setting $C'=(C\setminus \{y\})\cup \{x\}$,  $|\hf_{\mid C'}|=|\hf'_{\mid C}|\geq 13$, a contradiction again.
\end{proof}

There are two simple conditions to guarantee for a 4-set $C$ (with respect to a family $\hf$) that  $|\hf_{\mid C} |\leq 12$.

\begin{itemize}
  \item[(i)] $\exists \{x,y\}\in \binom{C}{2}$ such that no $F\in \hf$ contains $\{x,y\}$.
  \item[(ii)] $\exists \{x,y\}\in \binom{C}{2}$ such that $\hf(x)=\hf(y)$.
\end{itemize}
Note that if $\{x,y\}\subset F\in \hf$ then $F\setminus \{x\}\in \hf(x)$ but $F\setminus \{x\}\notin \hf(y)$. Thus (ii) implies (i).

In view of these conditions if $\{x,y\}\not\subset F$ for all $F\in \hf$ then we can symmetrize $\hf$ by removing all $F\in \hf$  with $y\in F$ and adding all $\{y\}\cup G$ with $G\in \hf(x)$. Thereby $\hf(x)=\hf(y)$ for the new family. By Claim \ref{claim-3.1} the new family preserves the property $\hf\not\rightarrow (4,13)$.  If $|\hf(x)|\geq |\hf(y)|$ then the new family has at least as many members as the old one. Thus we may assume that for all distinct $x,y\in [n]$ either $\exists F\in \hf$ with $\{x,y\}\subset F$ or $\hf(x)=\hf(y)$.

It is easy to see that $\hf(x)=\hf(y)$ is an equivalence relation. Thus we get a partition $[n]=Z_1\cup Z_2\cup \ldots\cup Z_r$   and an auxiliary family $\hh\subset 2^{[r]}$ such that each $Z_i$ is an equivalence class,  $F\in \hf$ iff $|F\cap Z_i|\leq 1$ for all $i$ and $\{i\colon F\cap Z_i\neq \emptyset\}\in \hh$. Let us choose $\hf$ such that $r$ is minimal over all families $\hf$ with $\hf\not\rightarrow (4,13)$ and $|\hf|$ maximal.

Note that $\hf(x)=\hf(y)$ forces that $x$ and $y$ are in the same $Z_i$. Hence if $1\leq i<i'\leq r$, $x\in Z_i$, $y\in Z_{i'}$ then $\{x,y\}\subset F$ for some $F\in \hf$. Consequently, $\binom{[r]}{2}\subset \hh$.

\begin{claim}\label{claim-2.4}
If $H,H'\in \hh^{(3)}$ then $|H\cap H'|\leq 1$.
\end{claim}
\begin{proof}
Suppose the contrary. WLOG $H=(1,2,3)$, $H'=(1,2,4)$. Since $(3,4)\in \hh$, $|\hf{\mid_C}|\geq 13$ for the corresponding $C=\{z_1,z_2,z_3,z_4\}$ (where $z_i\in Z_i$), a contradiction.
\end{proof}

Let $b_i=|Z_i|$, $i=1,2,\ldots,r$. If $r=3$, then the theorem follows from the fact that $(b_1+1)(b_2+1)(b_3+1)$ is maximized when $b_1=\lfloor\frac{n+2}{3}\rfloor$, $b_2=\lfloor\frac{n+1}{3}\rfloor$ and $b_3=\lfloor\frac{n}{3}\rfloor$. Thus in the rest of the proof we  assume $r\geq 4$.

\begin{claim}
For $x\in Z_i$, $y\in Z_j$ with $i\neq j$,
\begin{align}\label{ineq-key3}
|\hf(x,\bar{y})|< |\hf(y)|.
\end{align}
\end{claim}
\begin{proof}
If $|\hf(x,\bar{y})|\geq |\hf(y)|$ then for every $y\in Z_j$ we remove all edges $F\in \hf$ with $y\in F$ and add the edges $\{y\}\cup G$ for $G\in \hf(x,\bar{y})$. By Claim \ref{claim-3.1} the new family $\hf'$ satisfies $\hf'\not\rightarrow (4,13)$ and  $|\hf'|\geq |\hf|$. However, $\hf'$ has $r-1$ classes, contradicting the minimality of $r$.
\end{proof}

\begin{claim}\label{claim-3.4}
There exists $z\in [n]$ such that
\begin{align}\label{ineq-key6}
|\hf^{(3)}(z)| > \frac{n^2}{9}-\frac{n}{2}+1.
\end{align}
\end{claim}
\begin{proof}
Let $z\in [n]$ be a vertex with  $|\hf^{(3)}(z)|$ maximal. Note that $\hf=\hf^{(3)}\cup \hf^{(2)}\cup \hf^{(1)}\cup\hf^{(0)}$. Since $|\hf^{(0)}|+|\hf^{(1)}|=n+1$
 and $|\hf^{(2)}|\leq \binom{n}{2}$,
 \[
 |\hf^{(3)}| =|\hf|-|\hf^{(0)}|-|\hf^{(1)}|-|\hf^{(2)}|\geq \left\lfloor\frac{n+3}{3}\right\rfloor\left\lfloor\frac{n+4}{3}\right\rfloor\left\lfloor\frac{n+5}{3}\right\rfloor -n-1-\binom{n}{2}.
 \]
 It follows that
\begin{align*}
|\hf^{(3)}(z)| \geq \frac{ 3|\hf^{(3)}| }{n}\geq \frac{3}{n} \left\lfloor\frac{n+3}{3}\right\rfloor\left\lfloor\frac{n+4}{3}\right\rfloor\left\lfloor\frac{n+5}{3}\right\rfloor
-\frac{3n}{2}-\frac{3}{2}-\frac{3}{n}.
\end{align*}
For $n=3t$,
\[
|\hf^{(3)}(z)| \geq\frac{3}{n}\frac{(n+3)^3}{27}
-\frac{3n}{2}-\frac{3}{2}-\frac{3}{n} =  \frac{n^2}{9} -\frac{n}{2}+\frac{3}{2}.
\]
For $n=3t+1$ and $n\geq 8$,
\begin{align*}
|\hf^{(3)}(z)| \geq\frac{3}{n}\frac{(n+2)^2(n+5)}{27}
-\frac{3n}{2}-\frac{3}{2}-\frac{3}{n} &=  \frac{n^2}{9} -\frac{n}{2}+\frac{7}{6}-\frac{7}{9n}>\frac{n^2}{9}-\frac{n}{2}+1.
\end{align*}
For $n=3t+2$ and $n\geq 8$,
\begin{align*}
|\hf^{(3)}(z)| \geq\frac{3}{n}\frac{(n+1)(n+4)^2}{27}
-\frac{3n}{2}-\frac{3}{2}-\frac{3}{n} &=  \frac{n^2}{9} -\frac{n}{2}+\frac{7}{6}-\frac{11}{9n}>\frac{n^2}{9}-\frac{n}{2}+1.
\end{align*}
\end{proof}

 Let $x\in Z_i$ and assume $z\in Z_j$. If $i=j$ then clearly $\hf^{(3)}(x)=\hf^{(3)}(z)$. If $i\neq j$ then $|\hf(z,x)|\leq n-b_i-b_j+1$. By \eqref{ineq-key3},
\[
|\hf(z)|-|\hf(x)|\leq |\hf(z,x)|+|\hf(z,\bar{x})|-|\hf(x)| \leq |\hf(x,z)|-1\leq n-b_i-b_j.
\]
Since $|\hf^{(2)}(x)|=n-b_i$ and $|\hf^{(2)}(z)|=n-b_j$,
\begin{align*}
|\hf^{(3)}(z)|-|\hf^{(3)}(x)|&=(|\hf(z)|-|\hf^{(2)}(z)|-1)-(|\hf(x)|-|\hf^{(2)}(x)|-1)\nonumber\\[5pt]
&\leq n-b_i-b_j+(b_j-b_i)\\[5pt]
&=n-2b_i.
\end{align*}
By \eqref{ineq-key6} and $n\geq 25$, it follows that for all $x\in [n]$
\begin{align}\label{ineq-key7}
|\hf^{(3)}(x)| \geq |\hf^{(3)}(z)|-(n-2b_i)> \frac{n^2}{9}-\frac{3n}{2}+1+2b_i\geq \frac{n^2}{9}-\frac{3n}{2}+3\geq \frac{n^2}{18}.
\end{align}

If $r=4$, then by  Claim \ref{claim-2.4} we may assume that $\hh^{(3)}=\{(1,2,3)\}$. Then $\hf^{(3)}(x)=\emptyset$ for  all $x\in Z_4$, contradicting \eqref{ineq-key7}.

Let us fix $x_i\in Z_i$, $i=1,2,\ldots,r$.

\begin{claim}\label{claim-3.5}
$r\neq 5$.
\end{claim}
\begin{proof}
By Claim \ref{claim-2.4} and symmetry, we may assume that $\hh^{(3)}\subset\{(1,2,3),(1,4,5)\}$, $b_2\geq b_3$ and $b_4\geq b_5$.
Then
\[
|\hf|=b_1(b_2b_3+b_4b_5) +\sum_{1\leq i<j\leq 5}b_ib_j+n+1.
\]
Let $\hf'$ be the family obtained from $\hf$ by merging $Z_2$ and $Z_5$, $Z_3$ and $Z_4$. Then
\[
|\hf'|=b_1(b_2+b_5)(b_3+b_4)+b_1(b_2+b_5+b_3+b_4)+(b_2+b_5)(b_3+b_4)+n+1.
\]
Using  $b_1\geq 1$, we obtain that
\[
|\hf'|-|\hf| =b_1(b_2b_4+b_3b_5)-b_2b_5-b_3b_4\geq (b_2-b_3)(b_4-b_5)\geq 0.
\]
Clearly $\hf'\not\rightarrow (4,13)$ and $\hf'$ is 3-partite. This contradicts the minimality of $r$.
\end{proof}

\begin{claim}\label{claim-3.6}
$r\neq 6$.
\end{claim}
\begin{proof}
If there are two disjoint edges in $\hh^{(3)}$, then by Claim \ref{claim-2.4} $|\hh^{(3)}|=2$. Without loss of generality, assume that $\hh^{(3)}=\{(1,2,3),(4,5,6)\}$ and $b_4+b_5+b_6\leq \frac{n}{2}$. Then by \eqref{ineq-key7}
\begin{align}\label{ineq-key8}
|\hf^{(3)}(x_4)|+|\hf^{(3)}(x_5)|+|\hf^{(3)}(x_6)|=b_5b_6+b_4b_6+ b_4b_5 >\frac{n^2}{6}.
\end{align}
By \eqref{ineq-key0}, we infer that
\[
b_4b_5+b_4b_6+b_5b_6 \leq  \frac{(b_4+b_5+b_6)^2}{3} \leq \frac{n^2}{12},
\]
contradicting \eqref{ineq-key8}.

Thus $|H\cap  H'|=1$ for all $H,H'\in \hh^{(3)}$.  Up to isomorphism there is only one triple-system with four triples on six vertices. By symmetry we may assume that
\[
\hh^{(3)}\subset\{(1,3,5),(1,4,6),(2,3,6),(2,4,5)\}.
\]
Then
\begin{align*}
&b_1b_5+b_2b_6\geq |\hf^{(3)}(x_3)|> \frac{n^2}{18}, &b_1b_6+b_2b_5\geq |\hf^{(3)}(x_4)| > \frac{n^2}{18},\\[5pt]
&b_1b_3+b_2b_4\geq |\hf^{(3)}(x_5)|> \frac{n^2}{18}, &b_1b_4+b_2b_3 \geq |\hf^{(3)}(x_6)|> \frac{n^2}{18}.
\end{align*}
Adding these inequalities, we get
\begin{align}\label{ineq-3.1}
(b_1+b_2)(b_3+b_4+b_5+b_6) >\frac{2n^2}{9}.
\end{align}
Moreover,
\begin{align*}
&b_3b_5+b_4b_6\geq |\hf^{(3)}(x_1)|> \frac{n^2}{18}, &b_3b_6+b_4b_5\geq |\hf^{(3)}(x_2)| > \frac{n^2}{18}.
\end{align*}
It implies that
\begin{align}\label{ineq-3.2}
(b_3+b_4)(b_5+b_6) >\frac{n^2}{9}.
\end{align}
Note that $b_1+b_2+b_3+b_4+b_5+b_6=n$. If $b_1+b_2\geq \frac{n}{3}$, then $b_3+b_4+b_5+b_6\leq \frac{2n}{3}$. It follows that $(b_3+b_4)(b_5+b_6) \leq \frac{n^2}{9}$, contradicting \eqref{ineq-3.2}. If $b_1+b_2< \frac{n}{3}$, then
\[
(b_1+b_2)(b_3+b_4+b_5+b_6)<\frac{2n}{9},
\]
contradicting \eqref{ineq-3.1}.
\end{proof}

For $H\in 2^{[r]}$, let $b_H=\prod_{i\in H} b_i$.

\begin{claim}
$r=7$.
\end{claim}
\begin{proof}
Suppose that $r\neq 7$. Then by Claims \ref{claim-3.5} and \ref{claim-3.6}, $r\geq 8$. For each $i=1,2,\ldots,r$,
\[
|\hf^{(3)}(x_i)| =\sum_{P\in \hh^{(3)}(i)} b_P>\frac{n^2}{18}.
\]
By Claim \ref{claim-2.4}, $\hf^{(3)}(x_i)\cap \hf^{(3)}(x_j)=\emptyset$ for all $1\leq i<j\leq r$. By \eqref{ineq-key0}, we obtain that
\begin{align}\label{ineq-key10}
\frac{rn^2}{18}< \sum_{1\leq i\leq r}|\hf^{(3)}(x_i)| \leq \sum_{1\leq i<j\leq r} b_ib_j \leq \frac{r-1}{2r}\left(b_1+b_2+ \ldots+ b_r\right)^2=\frac{r-1}{2r}n^2.
\end{align}
It follows that
\[
\frac{r}{9}< 1- \frac{1}{r},
\]
which leads to a contradiction for $r\geq 8$.
\end{proof}

Now we assume that $r=7$.

\begin{claim}\label{claim-2.5}
For $n\geq 17$, $\max\limits_{1\leq i\leq 7} b_i\leq n/2$.
\end{claim}
\begin{proof}
Assume that $b_1\geq b_2\geq \ldots \geq b_7$. By \eqref{ineq-key7}, for all $x\in [n]$
\[
|\hf^{(3)}(x)| >  \frac{n^2}{9}-\frac{3n}{2}+3.
\]
It is easy to check that for $n\geq 17$ the RHS is greater than $\frac{n^2}{32}$. Assume that $b_1\geq \frac{n}{2}$ and we distinguish two cases.

{\bf \noindent Case 1. } $(1,2,3)\in \hh$.

Then
\[
|\hf^{(3)}(x_3)|\geq b_1b_2,\ |\hf^{(2)}(x_3)|=n-b_3,\ |\hf(x_1,x_3)|= b_2+1,
\]
and
\[
|\hf^{(3)}(x_1)| \leq b_2b_3+\left(\frac{n-b_1-b_2-b_3}{2}\right)^2,\ |\hf^{(2)}(x_1)|=n-b_1.
\]
By \eqref{ineq-key3}, $|\hf(x_3,\overline{x_1})|< |\hf(x_1)|$. It follows that
\[
b_1b_2+n-b_3-(b_2+1)\leq b_2b_3+n-b_1+\left(\frac{n-b_1-b_2-b_3}{2}\right)^2.
\]
Equivalently,
\begin{align}\label{ineq-newkey1}
b_2(b_1-b_3)+b_1-b_2-b_3-1\leq \left(\frac{n-b_1-b_2-b_3}{2}\right)^2.
\end{align}
Note that $b_1\geq \frac{n}{2}$ implies $b_1\geq b_2+b_3$. If $b_1=b_2+b_3$ then $b_1+b_2+b_3=n$ and \eqref{ineq-newkey1} cannot hold. Thus $b_1>b_2+b_3$. Then \eqref{ineq-newkey1} implies
\begin{align}\label{ineq-newkey2}
b_2(b_1-b_3)\leq \frac{n-b_1-b_2-b_3}{4}  (n-b_1-b_2-b_3).
\end{align}
If $b_2\geq \frac{n-b_1-b_2-b_3}{4}$, then \eqref{ineq-newkey2} implies
\[
b_1-b_3\leq n-b_1-b_2-b_3.
\]
It follows that $2b_1+b_2\leq n$, contradicting $b_1\geq \frac{n}{2}$. Thus $b_2< \frac{n-b_1-b_2-b_3}{4}$. That means $5b_2+b_3<n-b_1$. Then $6\frac{b_2+b_3}{2} \leq n-b_1\leq \frac{n}{2}$. It implies $\frac{b_2+b_3}{2}\leq \frac{n}{12}$. Therefore
\[
|\hf^{(3)}(x_1)|\leq \left(\frac{n}{12}\right)^2+\left(\frac{n}{12}\right)^2+\left(\frac{n}{12}\right)^2<\frac{n^2}{32},
\]
a contradiction.

{\bf \noindent Case 2. } $(1,2,3)\notin \hh$.

Then
\[
|\hf^{(3)}(x_1)|\leq b_2b_4+b_3b_5+b_6b_7.
\]
The maximum should be for $b_6=b_7=0$, $b_3=b_4$. Set $b_2=\alpha n$, $b_3=b_4=\beta n$, $b_5=\gamma n$, then
\begin{align}\label{ineq-newkey8}
\frac{|\hf^{(3)}(x_1)|}{n^2}\leq \alpha \beta +\beta \gamma =\frac{1}{2}[2\beta(\alpha+\gamma)] \leq \frac{1}{2}\left(\frac{2\beta+\alpha+\gamma}{2}\right)^2=\frac{1}{2}\left(\frac{n-b_1}{2n}\right)^2.
\end{align}
The RHS is at most $\frac{1}{2}\left(\frac{1}{4}\right)^2=\frac{1}{32}$ for $b_1\geq \frac{n}{2}$. Thus we get $|\hf^{(3)}(x_1)|\leq \frac{n^2}{32}$, a contradiction.
\end{proof}

By \eqref{ineq-key10}, we have
\begin{align}\label{ineq-key11}
\sum_{1\leq i\leq 7} |\hf^{(3)}(x_i)| \leq \frac{3n^2}{7}.
\end{align}
By \eqref{ineq-key6},
\[
|\hf^{(3)}(z)| > \frac{n^2}{9}-\frac{n}{2}+1.
\]
Assume $z\in Z_j$. Then Claim \ref{claim-2.5} implies $n-b_j\geq \frac{n}{2}$. By \eqref{ineq-key7}, we obtain that
\begin{align*}
\sum_{1\leq i\leq 7} |\hf^{(3)}(x_i)| &= |\hf^{(3)}(z)|+\sum_{i\neq j} \left(\frac{n^2}{9}-\frac{3n}{2}+1+2b_i\right)\\[5pt]
&> \frac{n^2}{9}-\frac{n}{2}+1+ 6\left(\frac{n^2}{9}-\frac{3n}{2}+1\right)+2(n-b_j)\\[5pt]
&\geq \frac{7n^2}{9}-\frac{19n}{2}+7+n\\[5pt]
&\geq \frac{7n^2}{9}-\frac{17n}{2}+7.
\end{align*}
It is easy to check that the RHS is greater than $\frac{3n^2}{7}$ for $n\geq 24$, contradicting \eqref{ineq-key11}.
\end{proof}

\section{Other results for $(n,m)\rightarrow (4,b)$}

Let us introduce the general notation
\[
m(n,a,b)=\min\left\{m\colon (n,m)\rightarrow (a,b)\right\}.
\]
In this section we consider $m(n,4,b)$ for $b\leq 16$. With this notation the  Sauer-Shelah-Vapnik-Chervonenkis Theorem is equivalent to $m(n,4,16)=1+\sum\limits_{0\leq i\leq 3} \binom{n}{i}$.

For $5<b<16$, it is easy to see that all extremal families satisfying $\hf\not\rightarrow (4,b)$ span $[n]$, i.e., $\cup_{F\in \hf} F=[n]$. Hence $\binom{[n]}{\leq 1}\subset \hf$. This motivates us to introduce the following auxiliary definitions. Set $\tilde{\hf}=\hf^{(2)}\cup \hf^{(3)}$ and say that $\tilde{\hf}$ is {\it complete} if $\partial \hf^{(3)}\subset \hf^{(2)}$, that is, if $P\subset T\in \hf^{(3)}$ and $|P|=2$ then $P\in \hf^{(2)}$.

Let us introduce the notation $\tilde{\hf}\hookrightarrow (4,c)$ if there exists a 4-set $C$ with $|\hf^{(2)}\cap \binom{C}{2}|+|\hf^{(3)}\cap \binom{C}{3}|\geq c$. If $\hf$ is a down-set with $|F|\leq 3$ for all $F\in \hf$ and $\cup \hf=[n]$ then $\tilde{\hf}\hookrightarrow (4,c)$ is equivalent to $\hf\rightarrow (4,c+5)$.

Finally, for $1\leq c<11$ we introduce the notation
\[
\tilde{m}(n,4,c)=\min\left\{\tilde{m}\colon |\tilde{\hf}|\geq \tilde{m} \mbox{ implies } \tilde{\hf} \hookrightarrow (4,c) \mbox{ for a complete family }\tilde{\hf}\subset 2^{[n]} \right\}.
\]
Clearly, $\tilde{m}(n,4,c)=m(n,4,c+5)-n-1$.

Let $\hht(r,n)$ be a complete $r$-partite graph on $n$ vertices with each part of size $\lfloor \frac{n}{r} \rfloor$ or $\lceil \frac{n}{r} \rceil$ and let $t(r,n)$ be the number of edges in $\hht(r,n)$. We have the following results.

\begin{table}[H]
  \centering
  \caption{$\tilde{m}(n,4,c)$ for $1\leq c\leq 8$, $n\geq 5$}
  \begin{tabular}{c}
  \toprule
  $\tilde{m}(n,4,1)=1$\\[3pt]
  \midrule
  $\tilde{m}(n,4,2)=2$\\[3pt]
    \midrule
  $\tilde{m}(n,4,3)=\lfloor \frac{2}{3}n\rfloor+1$\\[3pt]
    \midrule
  $\left(\frac{n}{2}\right)^{3/2}+o(n^{3/2})\leq  \tilde{m}(n,4,4)\leq \frac{1}{2}n^{3/2}+O(n)$\\[3pt]
    \midrule
  $\tilde{m}(n,4,5)=\lfloor\frac{n^2}{4}\rfloor+1$\\[3pt]
    \midrule
  $\tilde{m}(n,4,6)=t(3,n)+1$\\[3pt]
    \midrule
  $\tilde{m}(n,4,7)=\binom{n}{2}+1$ for $n\neq 6$, $\tilde{m}(6,4,7) = 17$\\[3pt]
    \midrule
  $\tilde{m}(n,4,8)=\left\lfloor\frac{n+2}{3}\right\rfloor\left\lfloor\frac{n+1}{3}\right\rfloor\left\lfloor\frac{n}{3}\right\rfloor+1$ for $n\geq 25$\\[3pt]
  \bottomrule
  \end{tabular}
\end{table}

The cases $c=9, 10$ will be discussed later. To prove the above statements for each particular choice of $c$, we assume that $\tilde{\hf}=\hf^{(2)}\cup \hf^{(3)}$ is a complete family with $\tilde{\hf}\not\hookrightarrow (4,c)$.

Since $\hf^{(3)}\neq \emptyset$ forces $\tilde{\hf}\hookrightarrow (4,4)$, for the case $1\leq c\leq 4$ we may assume $\hf^{(3)}=\emptyset$. Then $\tilde{m}(n,4,c)=c$ is trivial for $c=1$ and $c=2$.

Let $c=3$. Consider $\hf^{(2)}$, a graph in which no four vertices span more than 2 edges. Thus $\hf^{(2)}$ has maximum degree at most two and  without a path or cycle of length three. Hence each connected component of $\hf^{(2)}$ is a single edge or a path of length two. Consequently, $|\hf^{(2)}|\leq \frac{2}{3}n$, proving $\tilde{m}(n,4,3)=\lfloor \frac{2}{3}n\rfloor+1$.

For $c=4$, $\hf^{(2)}$ is a graph that contains no subgraph on 4 vertices with 4 or more edges. Let $C_3^+$ be a triangle plus a pendant  edge. It follows that $\hf^{(2)}$ is $C_3^+$-free and $C_4$-free. Consequently if $\hf^{(2)}$ contains a triangle, then it is a connected component. It follows that each connected component of $\hf^{(2)}$ with at least 4 vertices is $\{C_3,C_4\}$-free.  For a given family $\mathscr{F}$ of graphs, let $ex(n, \mathscr{F})$ denote the maximum number of edges in an $n$-vertex graph which does not contain any member in $\mathscr{F}$ as its subgraph. The {\it Zarankiewicz number} $z(n, C_4)$  is the maximum number of edges in an $n$-vertex bipartite graph without containing a $C_4$. It is well known that $z(n, C_4)=\left(\frac{n}{2}\right)^{3/2}+o(n^{3/2})$ (see \cite{DHS}, \cite{FS}). Since bipartite graphs are $C_3$-free, we see that $ex(n, \{C_3,C_4\})\geq z(n, C_4)$.

Erd\H{o}s-R\'{e}nyi-S\'{o}s \cite{ERS} and Brown \cite{brown} showed  that $ex(q^2+q+1,C_4)\geq \frac{1}{2}q(q+1)^2$ for all prime powers $q$. F\"{u}redi \cite{Furedi1,Furedi2} proved that $ex(q^2+q+1,C_4)=\frac{1}{2}q(q+1)^2$ for all prime powers $q\geq 14$.  As
it is shown in \cite{KST} this implies $ex(n,C_4)=\frac{1}{2}n^{3/2}+O(n)$ all $n$. Thus, $ex(n, \{C_3,C_4\})\leq ex(n,C_4)=\frac{1}{2}n^{3/2}+O(n)$. These results imply that $\left(\frac{n}{2}\right)^{3/2}+o(n^{3/2})\leq \tilde{m}(n,4,4)\leq \frac{1}{2}n^{3/2}+O(n)$.

\begin{prop}
\[
\tilde{m}(n,4,5) = \left\lfloor\frac{n^2}{4}\right\rfloor+1.
\]
\end{prop}

\begin{proof}
Note that $\hht(2,n)\not\hookrightarrow (4,5)$. This shows that $\tilde{m}(n,4,5)\geq \lfloor\frac{n^2}{4}\rfloor+1$.

Let $\tilde{\hf}=\hf^{(2)}\cup \hf^{(3)}$ be a complete family satisfying $\tilde{\hf}\not\hookrightarrow (4,5)$.
We prove $|\tilde{\hf}|\leq \lfloor\frac{n^2}{4}\rfloor$ by induction on $n$. Clearly it holds for $n=4$. Now we assume that it holds for $4,5,\ldots,n-1$ and prove it for $n$. If there exists $F_0\in \tilde{\hf}$ with $|F_0|=3$, then for any $y\in [n]\setminus F_0$, $\{x,y\}\notin \tilde{\hf}$ for all $x\in F_0$. It follows that $|\tilde{\hf}|\leq \lfloor\frac{(n-3)^2}{4}\rfloor+4\leq \lfloor\frac{n^2}{4}\rfloor$. Thus we may assume that $\hf^{(3)}=\emptyset$.

If there are two triangles with a common edge in $\hf^{(2)}$,  let $C$ be the set of these 4 vertices. Then $|\tilde{\hf}_{\mid C}|\geq 5$, a contradiction. Thus $\hf^{(2)}$ contains no two triangles with a common edge.  Assume  $\hf^{(2)}$ contains a triangle, say $\{x_1,x_2,x_3\}$. Then each $y\in [n]\setminus \{x_1,x_2,x_3\}$ has at most one neighbor in $\{x_1,x_2,x_3\}$. Therefore,
\[
|\{P\in \hf^{(2)}\colon P\cap \{x_1,x_2,x_3\}\neq \emptyset\}| \leq 3+(n-3)=n.
\]
Using the induction hypothesis, it follows that $|\tilde{\hf}|\leq \lfloor\frac{(n-3)^2}{4}\rfloor+n\leq \lfloor\frac{n^2}{4}\rfloor$. Finally if $\hf^{(2)}$ is triangle-free, then by  Mantel's theorem \cite{M},  $|\tilde{\hf}|= |\hf^{(2)}|\leq \lfloor\frac{n^2}{4}\rfloor$.
\end{proof}

\begin{prop}
\[
\tilde{m}(n,4,6) = t(3,n)+1.
\]
\end{prop}
\begin{proof}
Clearly $\hht(3,n)\not\hookrightarrow (4,6)$. We see that $\tilde{m}(n,4,6)\geq t(3,n)+1$.

Let $\tilde{\hf}=\hf^{(2)}\cup \hf^{(3)}$ be a complete family satisfying $\tilde{\hf}\not\hookrightarrow (4,6)$.
We prove $|\tilde{\hf}|\leq t(3,n)$ by induction on $n$. Clearly it holds for $n=4$. Now assume that it holds for $4,5,\ldots,n-1$ and we prove it for $n$. If there exists $F_0=\{x_1,x_2,x_3\}\in \tilde{\hf}$, then by $\tilde{\hf}\not\hookrightarrow (4,6)$ for every $y\in [n]\setminus F$  at most one of $\{x_1,y\},\{x_2,y\},\{x_3,y\}$ is in $\tilde{\hf}$. It follows that
\[
 |\{P\in \hf^{(2)}\colon P\cap F_0\neq \emptyset\}|\leq 3+n-3=n.
\]
Note that  $\tilde{\hf}\not\hookrightarrow (4,6)$ implies $|F\cap F'|\leq 1$ for all distinct $F,F'\in \hf^{(3)}$. We infer that $\hf^{(3)}(x_i)$ is a matching. Let $\hg(x_i)=\hf^{(3)}(x_i)\cap \binom{[n]\setminus F_0}{2}$, $i=1,2,3$.
We claim that for $1\leq i<j\leq 3$, $\hg(x_i)$ and $\hg(x_j)$ are  disjoint. For otherwise since $\tilde{\hf}$ is complete, we shall find $y\in [n]\setminus F_0$ such that two of $\{x_1,y\},\{x_2,y\},\{x_3,y\}$ are in $\tilde{\hf}$, a contradiction. Hence $\hg(x_1)\cup \hg(x_2)\cup \hg(x_3)$ is a matching. Therefore,
\[
 |\{F\in \hf^{(3)}\colon F\cap F_0\neq \emptyset\}|\leq 1+\left\lfloor\frac{n-3}{2}\right\rfloor=\left\lfloor\frac{n-1}{2}\right\rfloor.
\]
By the induction hypothesis,
\begin{align}\label{ineq-4.2}
|\tilde{\hf}|\leq t(3,n-3)+n+\left\lfloor\frac{n-1}{2}\right\rfloor.
\end{align}
\begin{claim}
\begin{align}\label{ineq-4.1}
t(3,n)-t(3,n-3)=2n-3.
\end{align}
\end{claim}
\begin{proof}
Note that
\[
t(3,n) =\binom{n}{2}-\binom{\lfloor\frac{n}{3}\rfloor}{2}-\binom{\lfloor\frac{n+1}{3}\rfloor}{2}-\binom{\lfloor\frac{n+2}{3}\rfloor}{2}.
\]
Then
\begin{align*}
t(3,n)- t(3,n-3)&=\binom{n}{2}-\binom{n-3}{2}-\left(\left\lfloor\frac{n}{3}\right\rfloor-1+\left\lfloor\frac{n+1}{3}\right\rfloor-1
+\left\lfloor\frac{n+2}{3}\right\rfloor-1\right)\\[5pt]
&=(3n-6)-(n-3)=2n-3.
\end{align*}
\end{proof}
Since $n\geq 5$ implies $n+\lfloor\frac{n-1}{2}\rfloor\leq 2n-3$, by \eqref{ineq-4.2} and \eqref{ineq-4.1} we obtain $|\tilde{\hf}|\leq t(3,n)$.
Thus we may assume that $\hf^{(3)}=\emptyset$.

Since  $\tilde{\hf}\not\hookrightarrow (4,6)$ implies $\hf^{(2)}$ is $K_4$-free, by Tur\'{a}n's Theorem \cite{turan} $|\tilde{\hf}|= |\hf^{(2)}| \leq  t(3, n)$.
\end{proof}

\begin{prop}
 $\tilde{m}(n,4,7) = \binom{n}{2}+1$  for $n\neq 6$ and   $\tilde{m}(6,4,7) = 17$.
\end{prop}
\begin{proof}
Note that $\binom{[n]}{2}\not\hookrightarrow (4,7)$. It follows that $\tilde{m}(n,4,7)\geq \binom{n}{2}+1$.

For $n=6$, define $\hf^{(3)} =\{ \{1,3,5\}, \{1,4,6\}, \{2,3,6), \{2,4,5\}\}$, $\hf^{(2)}=\partial \hf^{(3)}$ and $\tilde{\hf}=\hf^{(2)}\cup \hf^{(3)}$. Then $\hf^{(2)}$ is a complete 3-partite graph on parts $\{1,2\}$, $\{3,4\}$ and $\{5,6\}$. Now every 4-set $C\subset [6]$ contains at least one full part and at most one edge in $\hf^{(3)}$. It follows that $|\tilde{\hf}_{\mid C}|\leq 1+(6-1)=6$. Thus  $\tilde{m}(6,4,7) \geq  4+12+1=17$.

Suppose that $\tilde{\hf}=\hf^{(2)}\cup \hf^{(3)}$ is a complete family of the maximal size satisfying $\tilde{\hf}\not\hookrightarrow (4,7)$.

\begin{claim}\label{claim-4.4}
For any $P=\{z_1,z_2\}\notin \hf^{(2)}$, $|\hf^{(3)}(z_1)|\leq \lfloor\frac{n-2}{2}\rfloor$ and $|\hf^{(3)}(z_2)|\leq \lfloor\frac{n-2}{2}\rfloor$ .
\end{claim}
\begin{proof}
Note that $\tilde{\hf}\not\hookrightarrow (4,7)$ implies $|F\cap F'|\leq 1$ for all distinct $F,F'\in \hf^{(3)}$. It follows that $\hf^{(3)}(x)$ is a matching for all $x\in [n]$.

Let $T_1,T_2,\ldots,T_r$ be the triples in $\tilde{\hf}$ that contain $z_1$. Since $P\notin \hf^{(2)}$, none of them contain $z_2$ and $T_1\setminus\{z_1\},\ldots,T_r\setminus\{z_1\}$ are pairwise disjoint. Hence $|\hf^{(3)}(z_1)|=r\leq \lfloor\frac{n-2}{2}\rfloor$. Similarly, $|\hf^{(3)}(z_2)|\leq \lfloor\frac{n-2}{2}\rfloor$.
\end{proof}

Let us construct a bipartite graph $\hb$ between $\hf^{(3)}$ and $\binom{[n]}{2}\setminus \hf^{(2)}$ by connecting $T\in \hf^{(3)}$ and $P\in \binom{[n]}{2}\setminus \hf^{(2)}$ iff $T\cap P\neq \emptyset$. Note that in this case $|T\cap P|=1$ by completeness of $\tilde{\hf}$.

For $x\notin T\in \hf^{(3)}$, $\tilde{\hf}\not\hookrightarrow (4,7)$ implies that at least one of the edges $\{x,y\}$, $y\in T$ is missing from $\hf^{(2)}$. Thus the degree of $T$ in $\hb$ is at least $n-3$. Should the maximum degree of $P\in \binom{[n]}{2}\setminus \hf^{(2)}$ in $\hb$ be at most $n-3$, $|\hf^{(3)}|\leq \binom{n}{2}-|\hf^{(2)}|$ and thereby $|\tilde{\hf}|\leq \binom{n}{2}$ would follow.

Assume next that $P=\{z_1,z_2\}\in \binom{[n]}{2}\setminus \hf^{(2)}$ and it has degree at least $n-2$. By Claim \ref{claim-4.4} $|\hf^{(3)}(z_i)|\leq \lfloor\frac{n-2}{2}\rfloor$, $i=1,2$. If $n$ is odd we infer $|\hf^{(3)}(z_1)|+|\hf^{(3)}(z_2)|\leq n-3$, a contradiction. The only remaining possibility is that $n$ is even and $\hf^{(3)}(z_i)$ is a perfect matching for $i=1,2$. We need only  one of them.

Let $\hf^{(3)}(z_1)=\{E_i\colon 1\leq i\leq \frac{n-2}{2}\}$. We claim that at least two of the possible four edges between $E_i$ and $E_j$ are missing from $\hf^{(2)}$. Indeed otherwise we fix $x\in E_j$ that is joined (in $\hf^{(2)}$) to both vertices of $E_i$. However this forces that $\{z_1,x\}\cup E_i$ span a $K_4$ in $\hf^{(2)}$ whence $\tilde{\hf} \hookrightarrow (4,7)$.

Consequently, together with $P$ there are at least $2\binom{\frac{n-2}{2}}{2}+1$ missing edges from $\hf^{(2)}$. As to $T\in \hf^{(3)}$, $T\cap P=\emptyset$ would force that $T\cup\{z_i\}$ spans a $K_4$ in $\hf^{(2)}$ and  $\tilde{\hf}\hookrightarrow (4,7)$. Thus $|\hf^{(3)}|=|\hf^{(3)}(z_1)|+|\hf^{(3)}(z_2)|\leq n-2$. For $n\geq 8$, $2\binom{\frac{n-2}{2}}{2}+1>n-2$ implies $|\hf^{(2)}|+|\hf^{(3)}|<\binom{n}{2}$ and we are done.

For $n=6$ we infer $|\tilde{\hf}|=|\hf^{(2)}|+|\hf^{(3)}|\leq \binom{6}{2}-3+4=16$.
\end{proof}

What remains are $m(n,4,14)$ and $m(n,4,15)$. These are closely related to the famous unsolved problems of Tur\'{a}n  on $3$-graphs: $K_4^{(3)}$ and $K_4^{(3)-}$, where $K_4^{(3)}$ denotes the complete 3-graph on 4 vertices and $K_4^{(3)-}$ denotes $K_4^{(3)}$ minus an edge.  For a $k$-graph $F$, let $ex_k(n, F)$ denote the maximum number of edges in an $n$-vertex $k$-graph which does not contain $F$ as a subgraph. It is well known that $\lim_{n\rightarrow \infty}\binom{n}{k}^{-1}ex_k(n,F)$ exists. It is called the  {\it Tur\'{a}n density} of $F$ and denoted by $\pi(F)$. Tur\'{a}n \cite{turan} proposed a construction showing that $\pi(K_4^{(3)})\geq \frac{5}{9}$. Chung and Lu \cite{CL} proved $\pi(K_4^{(3)})\leq\frac{3+\sqrt{17}}{12} \approx 0.593592\cdots$. By applying the flag algebra method invented by Razborov, Razborov \cite{Ra} showed    $\pi(K_4^{(3)})\leq 0.561666$. For $K_4^{(3)-}$, Frankl and F\"{u}redi \cite{FF0} proved that $\frac{2}{7}\leq \pi(K_4^{(3)-})\leq \frac{1}{3}$. In \cite{FV}, by using the  flag algebra method  Falgas-Ravry and Vaughan showed  $\pi(K_4^{(3)-})\leq  0.286889$.

Let us derive the formula for $m(n,4,14)$ and $m(n,4,15)$ from a more general statement. In analogy with $3$-graphs let $K_r^{(k)}$ and $K_r^{(k)-}$ denote the complete $k$-graph and   complete $k$-graph minus an edge on $r$ vertices, respectively.

\begin{prop}
Let $\hf\subset 2^{[n]}$ be a down-set. Then (i) and (ii) hold.
\begin{itemize}
  \item[(i)] $\hf\not\rightarrow (k+1,2^{k+1}-1)$ iff $\hf^{(k)}$ is $K_{k+1}^{(k)}$-free.
  \item[(ii)] $\hf\not\rightarrow (k+1,2^{k+1}-2)$ iff $\hf^{(k)}$ is $K_{k+1}^{(k)-}$-free.
\end{itemize}
\end{prop}

\begin{proof}
Since the proofs are almost identical let us show (ii) only. If $Y\in \binom{[n]}{k+1}$ spans $K_{k+1}^{(k)-}$ in $\hf$, then being a down-set forces $\binom{Y}{\ell}\subset \hf$ for all $0\leq \ell<k$. Hence $|\hf_{\mid Y}|\geq 2^{k+1}-2$.

On the other hand if $\hf$ is a $K_{k+1}^{(k)-}$-free down-set then $|F|\leq k$ for all $F\in \hf$ and $|\hf\cap \binom{Y}{k}|\leq \binom{k+1}{k}-2$ for all $Y\in \binom{[n]}{k+1}$. Thus $\hf\not\rightarrow (k+1,2^{k+1}-2)$.
\end{proof}

\begin{cor}
\begin{itemize}
  \item[(i)] $m(n,k+1,2^{k+1}-1)=1+\sum\limits_{0\leq \ell<k}\binom{n}{\ell}+ex_k(n,K_{k+1}^{(k)})$.
  \item[(ii)] $m(n,k+1,2^{k+1}-2)=1+\sum\limits_{0\leq \ell<k}\binom{n}{\ell}+ex_k(n,K_{k+1}^{(k)-})$.
\end{itemize}
\end{cor}

Let us close this paper by stating an old but attractive conjecture. Recall that $\hf$ is {\it antichain} if $F\subset F'$ never holds for distinct members $F,F'\in \hf$.

\begin{conj}[\cite{F89}]
Let $k$ be a non-negative integer, $n\geq 2k$. Suppose that $\hf\subset 2^{[n]}$ is an antichain with $\hf\not\rightarrow (k+1,2^{k+1})$. Then $|\hf|\leq \binom{n}{k}$.
\end{conj}

Let us note that the statement was proved in \cite{F89} for $k\leq 2$ and by Anstee and Sali \cite{AS} for $k=3$.

\end{document}